\documentclass[a4paper, 12pt]{amsart}

\usepackage[utf8]{inputenc}

\usepackage[T2A]{fontenc}
\usepackage{amsmath,amssymb,amsthm, mathtools}
\usepackage[a4paper,hmargin=2.5cm,vmargin=2.5cm]{geometry}
\usepackage[english]{babel}

\usepackage{hyperref}

\usepackage{amsthm}
\usepackage{amscd}
\usepackage{amsmath}
\usepackage{amsfonts}
\usepackage{amssymb}
\usepackage{graphicx}
\usepackage{amsopn}
\usepackage[usenames]{color}
\usepackage{tikz}
\usepackage{tikz-cd}
\usepackage{textcomp}
\usepackage{marginnote}
\usetikzlibrary{arrows,shapes,snakes,automata,backgrounds,petri}

\theoremstyle{theorem}
\newtheorem{theorem}{Theorem}

\newtheorem{lemma}[theorem]{Lemma}

\theoremstyle{definition}
\newtheorem{definition}[theorem]{Definition}

\newtheorem{remark}[theorem]{Remark}

\theoremstyle{definition}
\newtheorem*{notation}{Notation}

\DeclareMathOperator{\sym}{Sym}

\DeclareMathOperator{\charac}{char}

\def\C{\mathbb{C}}

\def\a{\alpha}

\numberwithin{theorem}{section}
\allowdisplaybreaks[3]

\begin{document}
%%
%% Title page
%%
\date{}
\author{Konstantin M. Posadskiy}
\title{On a family of non-coset 2-valued groups}
\address{
    HSE University
}
\email{kmposadskiy@hse.ru}
\maketitle{}
%%
%% ===========================================================================
%%

\begin{abstract}
Within an important class of discrete-time 2-valued dynamical systems on $\mathbb{C}$ defined by polynomials, we extract a family of systems induced by the action of a 2-valued group. We study this 2-valued group and its analogues obtained by replacing $\mathbb{C}$ with an arbitrary field $F$. We establish that all of them are double coset groups. For some of them we prove non-cosetness. In the case $F = \mathbb{C}$, this yields the first example of a topologically non-coset, non-discrete topological 2-valued group.
\end{abstract}

\section{Introduction}

\subsection{Results}\
An interesting question in the theory of multivalued dynamical systems (dynamics) concerns the possibility of realizing a multivalued dynamics via the action of a multivalued group (in earlier papers, the term ``a dynamics is integrable by means of a multivalued group'' was used instead of ``a dynamics is induced by the action of a multivalued group.''). A partial answer to this question for discrete multivalued dynamics was obtained in \cite{G-YA}. One class of examples of continuous multivalued dynamics consists of dynamics defined by polynomials. A polynomial $P(z,w)$ in two variables with leading coefficient of $w^n$ equal to $1$ defines a continuous $n$-valued dynamics on $\C$ that sends $z$ to the $n$-tuple of roots of $P(z,w)$. As follows from \cite{KP}, many 2-valued dynamics of this type cannot be realized via the action of a multivalued group.

One of the main results of the present paper is finding a family of continuous 2-valued dynamics on $\mathbb{C}$ of this type that can be induced by the action of a multivalued group. Specifically, consider the family of 2-valued dynamics on $\mathbb{C}$ defined by polynomials of the form
$$P(z, w) = w^2 + 2(a_1z + a_0)w + (-a_1z + a_0)^2,$$
where $a_1 \not= 0$. This is a family of $2$-valued dynamics $z \mapsto (a\sqrt{z} \pm b)^2$ where $a, b \in \C$, $a \not= 0$. To present a 2-valued group that induces the dynamics described above, we give a general construction of a family of 2-valued groups, which is of independent interest.

Let $R$ be a commutative ring with identity in which $2 \not= 0$. Denote the set of invertible elements of this ring by $R^\times$. They form a group under multiplication. Let $M$ be some $R$-module, let $H$ be some subgroup of $R^\times$ containing $-1$. Denote the involution $f \mapsto -f$ on $R, M, H$ by $\iota$.

\begin{notation}
    Throughout the paper, square brackets are used in two distinct standard meanings, distinguished by the number of elements inside. The notation $[a]$ with a single element denotes the equivalence class of $a$. The notation $[x_1, \dots, x_k]$ with $k \geqslant 2$ denotes the multiset consisting of the listed elements.
\end{notation}

Consider the set $\mathcal{X}_{R, M, H} = (M/\iota) \times (H/\iota)$.

\begin{theorem}
    \label{really_2val_gr}
    The set $\mathcal{X}_{R, M, H}$ with the $2$-valued operation 
    $$([a], [b]) * ([c], [d]) = \bigl[([a+bc], [bd]), ([a-bc], [bd])\bigr]$$
    is a $2$-valued group.
\end{theorem}

As a special case of this construction, we can take $(F / \iota) \times (F^\times / \iota)$, where $F$ is an arbitrary field. We denote this $2$-valued group $\mathcal{X}_{F, F, F^\times}$ by $\mathcal{X}_F$.

\begin{definition} \label{definition:action}
    An \textit{action} of an $m$-valued group $A$ with unit $e$ and multiplication operation $\mu$ on a set $S$ is a mapping $\nu: A \times S \to \sym^m(S)$ such that
    
    1) The $m^2$-element multisets $\nu(a_1, \nu(a_2, s))$ and $\nu(\mu(a_1, a_2), s)$ coincide for all $a_1, a_2 \in A$, $s \in S$
    
    2) $\nu(e, s) = [s, \dots, s]$ for each $s \in S$.

    We say that an $m$-valued dynamics $T \in \mathcal{T}_m(S)$ \textit{is induced by the action of a 2-valued group $A$ with an element $a$} (in earlier papers ``integrable via $A$'') if there exists an action $\nu$ of $A$ on $S$ such that the multisets $T(s)$ and $\nu(a, s)$ coincide for each $s \in S$.
\end{definition}

In section \ref{section:2-val} we prove the following theorem about the $2$-valued group $\mathcal{X}_\C$.

\begin{theorem} \label{theorem:integrable_by_X}
    The 2-valued dynamics $z \mapsto (a\sqrt{z}\pm b)^2$ on $\C$ is induced by a continuous action of the $2$-valued group $\mathcal{X}_\C$ with the element $([a], [b])$.
\end{theorem}

\begin{remark}
    Of particular interest is the fact that all dynamics in this family are induced by the action of the same (non singly generated) 2-valued group.
\end{remark}

The question of whether a $2$-valued group $\mathcal{X}_{R, M, H}$ is a coset group and whether it is a double coset group is of independent interest.
\begin{definition}
    Let $G$ be a group, and let $\psi: H \to \operatorname{Aut}(G)$ be a homomorphism from an group $H$ of order $n$ to the automorphism group of $G$. The set of orbits $G/\psi(H)$ of elements of $G$ under the action of $H$, equipped with the operation
$$[g_1] * [g_2] = [g_1 \psi(h)(g_2)]_{h \in H},$$
is called the \textit{coset group} $G/\psi(H)$.
\end{definition}

\begin{definition}
    Let $G$ be a group, and let $H$ be its subgroup of order $n$. The \textit{double coset group} $H \backslash G / H$ is the set of classes $\{ HgH \mid g \in G \}$ equipped with the $n$-valued operation
    $$Hg_1H * Hg_2H = [Hg_1hg_2H]_{h \in H}$$
\end{definition}

Most known multivalued groups arise from the coset and double coset constructions; however, not all multivalued groups are coset groups. Thus, an important problem in the theory of multivalued groups is to determine whether a given multivalued group is isomorphic to some coset or to some double coset multivalued group. The remarkable survey paper \cite{mval} on the theory of multivalued groups, as well as \cite{BVEP}, \cite{Pon}, \cite{GPV}, present and study an example of a family of three-element multivalued groups that are not coset groups. Further results on the coset problem have been obtained in \cite{Lindef}, \cite{sigma}, \cite{BVG}, where the authors found new families of examples of non-coset commutative multivalued groups.

However, no examples of non-coset topological multivalued groups were known. The following theorem, proved in section \ref{section:non-coset}, addresses this question.

\begin{theorem} \label{theorem:non_topol_coset}
    The 2-valued group $\mathcal{X}_\C$ cannot be obtained from a topological group via the coset construction.
\end{theorem}

We also prove two following theorems for a $2$-valued group $\mathcal{X}_F$.

\begin{theorem} \label{theorem:non_coset_4k+1}
    Let $F$ be an arbitrary field of characteristic $4k+1$. Then the $2$-valued group $\mathcal{X}_F = (F / \iota) \times (F^\times/ \iota)$ is not coset
\end{theorem}

\begin{theorem} \label{theorem:bicoset}
    Let $F$ be an arbitrary field of characteristic $\not= 2$. Then the $2$-valued group $\mathcal{X}_F = (F / \iota) \times (F^\times/ \iota)$ is double coset.
\end{theorem}

\begin{remark}
    If the field $F$ is a topological field, then the corresponding topological $2$-valued group $\mathcal{X}_F$ is obtained via the double coset construction, again from a topological group.
\end{remark}

\subsection{Acknowledgements}\
The author is deeply grateful to his scientific advisor A. A. Gaifullin for his continued support,
inspiring discussions, and invaluable advice, which greatly improved the present manuscript. The author also sincerely thanks M. T. Urmanov, whose valuable comments and ideas improved this work.

\section{Two-valued group $\mathcal{X}_\C$ and two-valued dynamics} \label{section:2-val}

In this section we prove theorems \ref{really_2val_gr} and \ref{theorem:integrable_by_X}.

\begin{proof}[Proof of theorem \ref{really_2val_gr}.] 

    Let $M$ be an $R$-module, let $H$ be a subgroup of $R^\times$.
    Let us verify that the set $(M/\iota) \times (H/\iota)$ with the operation 
    $$([a], [b]) * ([c], [d]) = \bigl[([a+bc], [bd]), ([a-bc], [bd])\bigr]$$
    satisfies the multivalued group axioms.
    
    $1)\; ([0], [1]) * ([c], [d]) = [([c], [d]), ([c], [d])] = ([c], [d]) * ([0], [1])$
    
    $2) \;([a], [b]) * \left( \left[ \frac{a}{b} \right], \left[ \frac{1}{b} \right] \right) = [([a+a], [1]), ([0], [1])]$
    
    $\left(\left[\frac{a}{b}\right], \left[\frac{1}{b}\right]\right) * ([a], [b]) = \left[\left(\left[\frac{a}{b} + \frac{a}{b}\right], [1] \right), ([0], [1])\right]$

    $3) \; ([a], [b]) * (([c], [d]) * ([f], [g])) = ([a \pm b(c \pm df)], [bdg]) = (([a], [b]) * ([c], [d])) * ([f], [g])$
\end{proof}

\begin{proof}[Proof of theorem \ref{theorem:integrable_by_X}.]  

    Consider the map $\nu:A \times \C \to \sym^2(\C)$ such that 
    $$\nu(([g], [h]), z) = [(g\sqrt{z} \pm h)^2]$$
    Then
    \begin{align*}
    & 1) \;\nu(e, z) = \nu(([1], [0]), z) = [(\sqrt{z} \pm 0)^2] = [z, z] \\
    & 2) \;\nu\Big(([g], [h]), \nu\big(([j], [k]), z)\big)\Big) = \nu\Big(([g], [h]), (j\sqrt{z} \pm k)^2\Big) = (g(j\sqrt{z} \pm k) \pm h)^2 = \\
    & = (gj\sqrt{z} \pm gk \pm h)^2 = \nu\Big( [([gj], [gk+h]), ([gj], [gk-h])], z \Big) = \nu\Big(([g], [h])*([j], [k]), z)\Big)
    \end{align*}
\end{proof}

\section{Double cosetness} \label{section:bicoset}

In this section we prove theorem \ref{theorem:bicoset}.

\begin{proof}
    Let $F$ be an arbitrary field. Consider the two-valued group
    $$\mathcal{X}_F = (F / \iota) \times (F^\times/ \iota)$$ 
    with the operation
    $$([a], [b]) * ([c], [d]) = \bigl[([a+bc], [bd]), ([a-bc], [bd])\bigr].$$

    Denote the group of affine transformations of the line over the field $F$ by $G$; denote by $\varphi$ the affine transformation $y \mapsto -y$. Denote by $H$ the subgroup $\langle \varphi \rangle \subset G$. Let us show that
    $$(F / \iota) \times (F^\times/ \iota) \simeq H \backslash G / H$$

    We denote by $(a, b)$ the affine transformation $y \mapsto a + by$, where $a \in F$, $b \in F^\times$. Then multiplication in $G$ is given by the formula:
    $$(a, b)(c, d) = (a+bc, bd),$$
    thus multiplication in the double coset group $H \backslash G / H$ is given by the formula:
    $$H(a, b)H*H(c, d)H = [H(a, b)(c, d)H, H(a, b) \varphi(c, d)H] = [H(a+bc, bd)H, H(a-bc, -bd)H].$$

    The class $H(a,b)H$ of an arbitrary affine transformation $(a,b)$ in the double coset group $H \backslash G / H$ consists exactly of the affine transformations $y \mapsto \pm a \pm by$, i.e. $(\pm a, \pm b)$. Therefore, multiplication in $H \backslash G / H$ can be written as
    $$(\pm a, \pm b)*(\pm c, \pm d) = \bigl[(\pm(a+bc), \pm bd), (\pm(a-bc), \pm bd)\bigr]$$

    Thus, the map $(a, b) \mapsto H(a, b)H$ defines an isomorphism of the $2$-valued groups $(F / \iota) \times (F^\times/ \iota)$ and $H \backslash G / H$.
\end{proof}

\section{Non-cosetness} \label{section:non-coset}

We prove several general facts about the coset 2-valued groups $\mathcal{X}_F$, where $\charac F \not= 2$, that we later use to prove Theorems \ref{theorem:non_topol_coset} and \ref{theorem:non_coset_4k+1} in this section.

Suppose that the $2$-valued group $\mathcal{X}_F = (F/\iota) \times (F^\times /\iota)$ is coset, i.e. there exists a group $G$ and an involutive automorphism $\varphi \curvearrowright G$ such that $\mathcal{X}_F$ is isomorphic to the coset group $G/\langle \varphi \rangle$.

    Denote the $2$-valued subgroups $(F/\iota) \times \{[1]\}$ and $\{[0]\} \times (F^\times/\iota)$ of the $2$-valued group $\mathcal{X}_F$ by $X$ and $Z$, respectively. Denote by $\pi_G$ the projection $G \to G/\langle \varphi \rangle$ such that $g \mapsto [g] = \{g, \varphi(g)\}$ (the elements $g$ and $\varphi(g)$ may coincide). Denote the composition of the map $\pi_G$ and the isomorphism $iso: G / \langle \varphi \rangle \to Y$ by $\pi$.

    \[
    \begin{tikzcd}
      G \arrow[d, "\pi_G"] \arrow[dr, "\pi"] &  \\
      G / \langle \varphi \rangle \arrow[r, "iso"]   & \mathcal{X}_F
    \end{tikzcd}
    \]

    Denote by $G_X, G_Z \subset G$ the preimages of $2$-valued subgroups $X$ and $Z$ under the map $\pi$, respectively. Since $X$ and $Z$ are two-valued subgroups of $\mathcal{X}_F$, it follows that the groups $G_X$ and $G_Z$ are ordinary subgroups of $G$.

    We study the structure of the group $G$ and its subgroups $G_X$ and $G_Z$, as well as how the automorphism $\varphi$ acts on them.

    Take arbitrary $g_1, g_2 \in G$.
    By the definition of a coset multivalued group:
    \begin{align*}
    &\pi(g_1)*\pi(g_2) = \bigl[\pi(g_1 g_2), \pi(g_1 \varphi(g_2))\bigr] = \bigl[\{g_1g_2, \varphi(g_1g_2)\}, \{g_1\varphi(g_2), \varphi(g_1)g_2\}\bigr]
        \tag{CG}\label{eq:coset-rule}
    \end{align*}
    Thus the following lemma holds.
    \begin{lemma} \label{prop::projection}
        $\forall g_1, g_2 \in G \;\; \pi(g_1g_2) \in \pi(g_1)*\pi(g_2)$.    
    \end{lemma}

    \begin{lemma} \label{lemma:pi(xz)}
        Let $\pi(x) = ([a], [1])$, $\pi(z) = ([0], [b])$. Then $\pi(xz) = ([a], [b])$.
    \end{lemma}

    \begin{proof}
        From lemma \ref{prop::projection} it follows that
        $$\pi(xz) \in \pi(x) * \pi(z) = [([a], [b]), ([a], [b])],$$
        from where the statement of the lemma follows.
    \end{proof}

    \begin{lemma} \label{prop::phi_structure}
    The authomorphism $\varphi$ has the following properties:
    
    1) $\forall x \in G_X \setminus \{e\} \;\; \varphi(x) \not= x$

    2) $\forall z \in G_Z \;\; \varphi(z) = z$
    \end{lemma}

    \begin{proof}
    Suppose that there exists an element $x \in G_X$ such that $\varphi(x) = x$. Denote $\pi(x)$ by $([a], 1)$. Then
    $$([a\pm a], 1) = \pi(x)*\pi(x) = [[x^2], [x\varphi(x)]] = [[x^2], [x^2]]$$
    Since the characteristic of the field $F$ is not $2$, it follows that $a = 0$ and thus $x = e$. Therefore $\forall x \in G_X \setminus \{e\} \;\; \varphi(x) \not= x$.

    For all $x \in G_X$, $z \in G_Z$ from \eqref{eq:coset-rule} and lemma \ref{lemma:pi(xz)} it follows that $\{xz, \varphi(x z)\} = \{x\varphi(z), \varphi(x)z\}$. The $2$-valued subgroup $X$ contains more than one element. Hence there exists an element $x \in G_X$ such that $\varphi(x) \not= x$. Therefore $\forall z \in G_Z$ $xz \not=\varphi(x)z$; thus $xz = x \varphi(z)$. This means that $z = \varphi(z)$ for all $z \in G_Z$.
    \end{proof}

    \begin{lemma} \label{lemma:elements_representation}
    Every element of the group $G$ admits a unique representation as the product of an element of the subgroup $G_X$ and an element of the subgroup $G_Z$.
    \end{lemma}
    \begin{proof}
        1) Existance. Consider an arbitrary element $g \in G$. Denote its image under the map $\pi$ by $([a], [b])$. If $[a] = 0$, then $\pi(g) \in Z$, therefore $g \in G_Z$. Then $g$ can be represented as $e \cdot g$. If $[a] \not= 0$, then the element $([a], [1]) \in X$ has two preimages in $G_X$ under $\pi$, namely some $x$ and $\varphi(x)$, while the element $([0],[b])$ has a single preimage; denote it by $z$. It follows from lemma \ref{lemma:pi(xz)} that both elements $xz$, $\varphi(x)z$ map to $([a], [b])$ under the mapping $\pi$. Since every element of $\mathcal{X}_F$ has no more than $2$ preimages, it follows that one of the products $xz$, $\varphi(x)z$ equals $g$.

        2) Uniqueness. Let $x_1, x_2 \in X$, $z_1, z_2 \in Z$, $x_1 z_1 = x_2 z_2$. Denote by $([a_1], [1])$, $([a_2], [0])$ the images of $x_1$, $x_2$ under $\pi$, respectively. Denote by $([0], [b_1])$, $([0], [b_2])$ the images of $z_1$, $z_2$ under $\pi$, respectively.

        It follows from lemma \ref{lemma:pi(xz)} that $\pi(x_1z_1) = ([a_1], [b_1])$, $\pi(x_2z_2) = ([a_2], [b_2])$. Since $x_1z_1 = x_2z_2$, we have $[a_1] = [a_2]$, $[b_1] = [b_2]$. It follows from lemma \ref{prop::phi_structure} that the element $([0], [b_1])$ has exactly one preimage under $\pi$. Thus $z_1 = z_2$. Therefore, since $x_1z_1 = x_2z_2$, we also obtain $x_1 = x_2$.
    \end{proof}

    We now prove the following key lemma.

    \begin{lemma} \label{lemma:big_cosetity}
        Let $F$ be a field of characteristic $\not= 2$. Suppose that the $2$-valued group $\mathcal{X}_F = (F / \iota) \times (F^\times / \iota)$  with the operation
        $$([a], [b]) * ([c], [d]) = \bigl[([a+bc], [bd]), ([a-bc], [bd])\bigr]$$
        is isomorphic to a coset group $G/\langle \varphi \rangle$. Then

        1) There exists a bijection
        $$\alpha: G \longrightarrow F \times (F^\times/ \iota)$$
        under which the operation on the set $F \times (F^\times/ \iota)$ given by the formula
        $$x \cdot y = \alpha(\alpha^{-1}(x)\alpha^{-1}(y))$$
        satisfies the following condition: for any $a, c \in F$, $b, d \in F^\times$ the equality
        $$(a, [b])\cdot(c, [d]) = (\pm a \pm bc, [bd])$$
        holds for some choice of signs.

        2) If $F = \C$ and the group $G$ is topological, then there exists a homeomorphism $\alpha$ satisfying this condition.
    \end{lemma}

    \begin{proof}[Proof of statement (1) of lemma \ref{lemma:big_cosetity}.]

        The projection $\pi$ defines a bijection $\alpha_Z$ between $G_Z$ and $Z = F^\times/ \iota$. Since every nonzero element of $X$ has exactly two preimages under the projection $\pi$, and $iso: G / \langle \varphi \rangle \to (F/\iota) \times (F^\times / \iota)$ defines a bijection between $G_X/ \langle \varphi \rangle$ and $X = F/ \iota$, we obtain that there also exists a bijection $\alpha_X$ between $G_X$ and $F$ such that if $\pi(x) = ([a], [1])$, then the pair of elements $x$, $\varphi(x) \in G_X$ maps under $\alpha_X$ to the pair of elements $\pm a$.

        We now extend $\alpha_X$ and $\alpha_Z$ to the whole group $G$. From lemma \ref{lemma:elements_representation} it follows that for every $g \in G$ there exist unique elements $x \in X$, $z \in Z$ such that $xz = g$. Set $\alpha(g) = (\alpha_X(x), \alpha_Z(z))$. This rule defines a bijection.

        We now prove that this bijection satisfies the condition of the lemma. Let $g_1, g_2$ be arbitrary elements of the group $G$; let $g_1 = x_1z_1$, $g_2 = x_2z_2$ where $x_1, x_2 \in G_X$, $z_1, z_2 \in G_Z$. Denote by $([a], [1])$, $([0], [b])$, $([c], [1])$, $([0], [d])$ the images of $x_1, z_1, x_2, z_2$ under $\pi$, respectively. It follows from lemma \ref{lemma:pi(xz)} that $\pi(g_1) = ([a], [b])$, $\pi(g_2) = ([c], [d])$. By \eqref{eq:coset-rule}, $\pi(g_1g_2) \in \{([a\pm bc], [bd])\}$. Therefore $\alpha(g_1g_2) = (\pm a \pm bc, [bd])$ for some choice of signs as required.
    \end{proof}

    To prove the second part of lemma \ref{lemma:big_cosetity} we need the following standard lemma; see, e.g., \cite[I.3, th. 3.1]{Bredon}.
    \begin{lemma} \label{lemma:factor_is_proper}
        The quotient map of a Hausdorff space by an involutive homeomorphism is proper.
    \end{lemma}

    We now prove the second part of lemma \ref{lemma:big_cosetity}.
    
    \begin{proof}[Proof of statement (2) of lemma \ref{lemma:big_cosetity}.]

        Since the projection $\pi: G \to (\C / \iota) \times (\C^\times / \iota)$ is continuous in this case, it follows that the resulting bijection $\alpha_Z$ is also continuous. Therefore, for the bijection $\a$ to be continuous in the case $F = \C$, it is necessary and sufficient that the bijection $\alpha_X$ described in statement (1) be continuous.

        The projection $\pi$ defines a continuous two-sheeted covering $p: G_X \setminus \{e\} \to (\C \setminus \{0\}) / \iota$. Let us show that this covering is nontrivial. Suppose the contrary. Then $$G_X \setminus \{e\} = (\C \setminus \{0\}) / \iota \sqcup (\C \setminus \{0\}) / \iota.$$ We denote the class $[z]$ from one of these copies by $[z]_1$, and the class of the same number $z$ from the other copy by $[z]_2$, having $\pi([z]_1) = \pi([z]_2) = [z]$. It follows from equality \eqref{eq:coset-rule} that
        $$[z]_1 \cdot [z']_1, [z]_1 \cdot [z']_2 \mapsto [z \pm z'].$$ 
        Fix $z = 2$ and $j \in \{1, 2\}$ such that $[2]_1 \cdot [1]_j \xmapsto{p} [3]$. Continuity of the operation implies that for any class $[z']_j$ we have $[2]_1 \cdot [z']_j \mapsto [2 + z']$.

        Consider the path $z' = e^{i \pi t}$, $t \in [0,1]$. Since $[e^0]_j = [e^{i\pi}]_j$, continuity of the operation forces $[3] = [2+e^0] = [2 + e^{i \pi}] = [1]$, which is false. Hence, the projection $\pi$ defines a nontrivial continuous two-sheeted covering $G_X \setminus \{e\} \to (\C/\iota) \setminus \{0\}$.

        It can be lifted to a continuous map $G_X \setminus\{e\} \to \C \setminus \{0\}$. Its extension by zero at the point $e$ is a continuous bijection $\alpha_X$, which yields a continuous bijection $\alpha: G \to \C \times (\C^\times / \iota)$. This bijection, together with the natural projection $g: \C \times (\C^\times / \iota) \to (\C / \iota) \times (\C^\times / \iota)$, completes the following commutative diagram.

        \[
        \begin{tikzcd}
          G \arrow[r, "\alpha"] \arrow[d, "\pi_G"] \arrow[dr, "\pi"] & \C \times (\C^\times / \iota) \arrow[d, "g"] \\
          G/H \arrow[r, "iso"]                              & (\C / \iota) \times (\C^\times / \iota)
        \end{tikzcd}
        \]

        Let us prove that the resulting map $\alpha$ is a homeomorphism.

        The map $\pi_G$ is proper by lemma \ref{lemma:factor_is_proper}. Since $iso$ is an isomorphism, it follows that $\pi$ is also proper. Let $K \subset \C \times (\C^\times / \iota)$ be a compact set. Then, since $g$ is continuous and $\pi_G$ is proper, it follows that $\pi^{-1}(g(K))$ is also a compact set. Continuity of $\alpha$ implies that $\alpha^{-1}(K)$ is a closed subset of the compact set $\pi^{-1}(g(K))$, hence it is compact itself. Therefore, $\alpha$ is a continuous proper bijection from a Hausdorff space onto a locally compact space, which implies that it is a homeomorphism.
    \end{proof}

We now use the lemmas to prove theorems \ref{theorem:non_coset_4k+1} and \ref{theorem:non_topol_coset}.

\begin{proof}[Proof of theorem \ref{theorem:non_coset_4k+1}.]
    Let $F$ be a field of characteristic $4k+1$. Suppose that the $2$-valued group
    $$\mathcal{X}_F = (F / \iota) \times (F^\times/ \iota)$$
    with the operation
    $$([a], [b]) * ([c], [d]) = \bigl[([a+bc], [bd]), ([a-bc], [bd])\bigr]$$
    is coset.

    Denote by $\mathbb{F}_{4k+1}$ the subfield of $F$ consisting of $4k+1$ elements.

    The group $F^\times/ \iota$ has an even number of elements, thus it has an element of order $2$. Denote it by $[s]$. Identify the elements of $G$ with the elements of $F \times (F^\times / \iota)$ via the bijection from lemma \ref{lemma:big_cosetity}.
    Consider the subgroup $A = \mathbb{F}_{4k+1} \times \{[1], [s]\} \subset G$. This subgroup consists of $2\cdot(4k+1)$ elements. Consider the element $(1, [s]) \in G$.

    It follows from lemma \ref{lemma:big_cosetity} that $(1, [s])\cdot(1, [s]) = (\pm 1\pm s, 1)$ which is not equal to $(0, 1)$ for any choice of signs. Hence, the order of the element $(1, [s])$ is not $2$. Moreover, $(1, [s])^{4k+1} = (\dots, [s]) \not= (0, 1)$. Therefore, the order of the element $(1, [s])$ is not $4k+1$ also. Thus, it equals $2 \cdot (4k+1)$. It follows that the subgroup $A \subset G$ is cyclic and hence commutative.

    Therefore
    $$(1, [1])\cdot(0, [s]) = (0, [s])\cdot(1, [1])$$
    On other hand, it follows from lemma \ref{lemma:big_cosetity} that
    $$(1, [1]) \cdot (0, [s]) = (\pm 1, [s]),$$
    $$(0, [s])\cdot (1, [1]) = (\pm s, [s]),$$
    a contradiction.
\end{proof}

\begin{proof}[Proof of theorem \ref{theorem:non_topol_coset}.]
    It follows from the second statement of lemma \ref{lemma:big_cosetity} that if the 2-valued group $\mathcal{X}_\mathbb{C}$ is obtained via the coset construction from a topological group $G$, then there exists a homeomorphism $\alpha: G \to \C \times (\C^\times/ \iota)$, under which the operation of the group $G$ is sent to an operation on the set $\C \times (\C^\times/ \iota)$ such that for all $a, c \in \C$, $b, d \in \C^\times$ the following equality holds for some choice of signs: 
    $$(a, [b])\cdot(c, [d]) = (\pm a \pm bc, [bd])$$
    Since $\alpha$ is a homeomorphism, $G$ is a topological group and $z_1 \cdot z_2 = \alpha^{-1}(\alpha(z_1)\cdot \alpha(z_2))$, it follows that this operation on the set $\C \times (\C^\times/ \iota)$ is continuous in both arguments.
    Therefore the set $\C \times (\C^\times/ \iota)$ carries the continuous operation with the following property:
    $$(a, [b]) \cdot (c, [d]) \in \{(\pm a \pm bc, [bd])\}$$
    Consider a continuous path on $\C \times (\C^\times/ \iota)$:
    $$I(t) = (0, [e^{i \pi t}])$$
    Continuity of multiplication implies that either for all $t \in [0,1]$ we have
    $$I(t) \cdot (1, [1]) = (e^{i \pi t}, [1])$$
    or for all $t \in [0,1]$ we have
    $$I(t) \cdot (1, [1]) = (-e^{i \pi t}, [1])$$

    It follows that the first element $e^{i \pi t}$ of the pair $I(t) \cdot (1, [1]) \in \C \times (\C^\times/ \iota)$ changes sign as $t$ runs through the interval $[0, 1]$. But $I(0) = I(1)$ because $[e^{i \pi t}] = \{e^{i \pi t}, -e^{i \pi t}\}$. A contradiction.
\end{proof}

\end{document}